\documentclass[11pt]{amsart}
\usepackage{amsmath, amssymb, amsfonts, amsthm, amscd, vmargin}
\usepackage[all]{xy}
\usepackage{graphicx,color}
\usepackage{enumerate}
\usepackage{enumitem}
\usepackage{tikz}
\usetikzlibrary{shapes}

\theoremstyle{plain}
\newtheorem{teo}{Theorem}
\newtheorem{lem}[teo]{Lemma}
\newtheorem{prop}[teo]{Proposition}

\theoremstyle{definition}
\newtheorem{defi}[teo]{Definition}

\newtheorem{rem}[teo]{Remark}

\newcommand{\Aut}{\operatorname{Aut}}

\newcommand{\Hom}{\operatorname{Hom}}

\newcommand{\Cbb}{{\mathbb C}}
\newcommand{\Qbb}{{\mathbb Q}}
\newcommand{\Zbb}{{\mathbb Z}}

\newcommand{\Pbb}{{\mathbb P}}

\newcommand{\Gbb}{\mathbb{G}}
\newcommand{\Sbb}{\mathbb{S}}

\begin{document}

\title{ Horospherical two-orbit varieties as zero loci}

\author{Boris Pasquier}
\address{Laboratoire de Math\'ematiques Appliqu\'ees de Poitiers, 
CNRS, Univ. Poitiers}
\email{boris.pasquier@univ-poitiers.fr}  
 
\author{Laurent Manivel}
\address{Institut de Math\'ematiques de Toulouse ; UMR 5219, Universit\'e de Toulouse \& CNRS, F-31062 Toulouse Cedex 9, France}
\email{manivel@math.cnrs.fr}


\date{\today}

\maketitle
 	
\begin{abstract}
We present geometric realizations of horospherical two-orbit varieties, by showing that their blow-up along the unique closed
invariant orbit is the zero locus of a general section of a homogeneous vector bundle over some auxiliary variety. As an application, 
we compute the cohomology ring of the $G_2$-variety, including its 
quantum version. We also consider the $Spin_7$-variety, which deserves a different treatment. 
\end{abstract}

\section{Introduction}

Homogeneous varieties play an important role in the classification of complex Fano manifolds, one of the main building blocks in the classification of complex projective varieties. Already in dimension three, the Fano-Iskovskih classification of Fano threefolds of Picard number one and index one reveals that many of them (those of genus between $6$ and $10$, to be precise) can be realized as complete intersections in 
certain homogeneous spaces \cite{IP}. In genus $12$, one has to consider 
an equivariant bundle over a Grassmannian in order to realize the 
Fano threefolds with this genus as zero loci of global sections. 
If other approaches are
also possible, this {\it vector bundle method} was applied 
systematically 
by Mukai and many others for Fano threefolds and K3 surfaces.
It  
has the 
great advantage of allowing an easy access to the geometry of these
varieties \cite{mukai-vb}. Very recently, the vector bundle method was 
used in order to cover the whole of Mori-Mukai's classification of Fano threefolds \cite{bef}. 

In higher dimensions, complete intersections and, more generally, zero loci of sections of homogeneous vector bundles on homogeneous varieties also allow to construct lots of interesting varieties (in dimension four, see \cite{kuchle, benedetti} for a sample of these
techniques). Nevertheless, it is certainly important to enlarge the 
class of ambient manifolds on which one could use the vector bundle 
method. Close to homogeneous varieties, one can consider  quasi-homogeneous varieties (those varieties whose automorphism
group acts with a dense orbit), especially those that have been
classified by combinatorial data, such as spherical varieties, 
or even more special ones, such as symmetric varieties or 
horospherical varieties. Under the hypothesis that the Picard 
group is cyclic (which implies that these varieties are Fano), 
symmetric varieties were classified by Ruzzi \cite{ruzzi}. They 
are in fact homogeneous, or hyperplane sections of homogeneous varieties, up to two exceptions. It is remarkable that these
two exceptional varieties can both be realized geometrically
by the 
vector bundle method; this was used in \cite{cg, dg} in order to
study their geometries and compute their (quantum) intersection rings. 

\medskip 
In this paper we consider the case of horospherical varieties.
Before stating our main result, 
let us recall the classification of smooth projective horospherical varieties with  Picard group~$\Zbb$.

\begin{teo}\cite[Th.~0.1]{2orbits}\label{2orbits}
Let $G$ be a connected reductive algebraic group. Let $X$ be a
smooth projective horospherical $G$-variety with Picard group~$\Zbb$, which is not homogeneous.

Then $X$ is horospherical of rank one, and its automorphism group is a connected non-reductive linear algebraic group, strictly containing $G$, acting with exactly two orbits.

Moreover, $X$ is uniquely determined by its two
closed $G$-orbits $Y$ and $Z$, isomorphic to $G/P_Y$ and $G/P_Z$
respectively. With the convention that $Z$ is fixed by $\operatorname{Aut}(X)$, the possible triples  $(G,P_Y,P_Z)$ 
are the following:
\begin{enumerate}
\item $(B_m,P(\varpi_{m-1}),P(\varpi_m))$ with $m\geq 3$
\item $(B_3,P(\varpi_1),P(\varpi_3))$
\item $(C_m,P(\varpi_{i+1}),P(\varpi_i))$ with $m\geq 2$ and
$i\in\{1,\ldots,m-1\}$
\item $(F_4,P(\varpi_2),P(\varpi_3))$
\item $(G_2,P(\varpi_2),P(\varpi_1))$
\end{enumerate}
\end{teo}

We denoted by $P(\varpi_i)$ the maximal
parabolic subgroup of $G$ corresponding to the dominant weight
$\varpi_i$, with the notations of Bourbaki \cite{Bourbaki456}.
We will also denote it $P_i$ for simplicity.

Note that for each group $G$ there is at most one variety in this list, that we will call 
the $G$-variety. We refer to \cite{GPPS} for more geometric information on these varieties 
(in particular their dimesion and their index). It turns out that the $Spin_7$-variety
has the special feature of being a (generic) hyperplane section
of the spinorial variety $Spin_{10}/P_5\subset\mathbb{P}(\Delta)$, 
where $\Delta$ denotes any of the half-spin representations of $Spin_{10}$. Since its 
dimension is $9$ and its index is $7$, this follows 
from Mukai's classification of Fano varieties of coindex three \cite{mukai-coindex3}).

\smallskip
Our main result provides geometric models of the remaining 
varieties. 

\begin{teo}\label{th:main}
Let $X$ be a
smooth projective horospherical $G$-variety with Picard group~$\Zbb$, which is not homogeneous, and with $G\ne Spin_7$. Let 
 $q:\tilde{X}\rightarrow X$ be the blow-up of $X$ along the closed orbit $Z$ fixed by $\Aut(X)$. Then $\tilde{X}$ can be realized as the zero locus of a general section of a vector bundle over some homogeneous space.
 
 More precisely, there exist a fundamental $G$-module $V$, and a positive integer $k$ , such that $\tilde{X}$ coincides with  
 the zero locus of a general section $s$ of the vector bundle $\mathcal{E}=\mathcal{Q}\boxtimes\mathcal{U}^*$ over $G/P_{k+1}\times \mathbb{G}(k+1,V\oplus\Cbb)$, where $\mathcal{Q}$ is the tautological quotient bundle over $G/P_{k+1}\subset \mathbb{G}(k+1,V)$, and $\mathcal{U}$ the tautological bundle over $\mathbb{G}(k+1,V\oplus\Cbb)$.
 
 Finally, the projection $p$ to $G/P_{k+1}$ realizes $\tilde{X}$
 as the projective bundle $\mathbb{P}(\Cbb\oplus \mathcal{V}^*)$, if $\mathcal{V}$ denotes the tautological rank $k+1$ bundle over $G/P_{k+1}$.
\end{teo}
We can illustrate the theorem by the following diagram.
 $$\xymatrix{ 
 & & \mathcal{Q}\boxtimes\mathcal{U}^* \ar@{->}[d]&  \\
 & \mathcal{Q}\ar@{->}[d]   & G/P_{k+1}\times \mathbb{G}(k+1,V\oplus\Cbb)\ar@/^-1pc/[u]_{s\; general} 
\ar@{->}[ld]\ar@{->}[rd] &  \mathcal{U}^* \ar@{->}[d] \\
  \mathbb{G}(k+1,V) &  G/P_{k+1} \ar@{_{(}->}[l] & 
    \tilde{X}
    \ar@{^{(}->}[u]\ar@{->}[ld]_p\ar@{->}[rd]^q
  & \mathbb{G}(k+1,V\oplus\Cbb) & \\
  &  G/P_{k+1}\ar@{=}[u] & & X \ar@{^{(}->}[u] & 
  }$$

\medskip
In the second part of the paper, we use these simple geometric models to improve our understanding of the cohomology (or Chow) ring of
the horospherical varieties. The Chevalley formulas for those varieties have been obtained in \cite{GPPS}, including the quantum version, which was recently used to prove that Galkin's Conjecture $\mathcal{O}$ does hold for these varieties \cite{conjO}. We give a complete treatment of the $G_2$-variety in Section 3. We also discuss the cohomology ring of the $Spin_7$-variety in Section 4 and extend 
these results to quantum cohomology. Unfortunately, although the Chow ring of the $F_4$-variety could in principle be determined following the 
same approach, the computational complexity seems too big for 
this case to be accessible. 

\section{Proof of Theorem~\ref{th:main}}

	Let $G$ be a simple algebraic group over $\Cbb$. Fix a Borel subgroup $B$ of $G$ containing a maximal torus $T$.

We will show that to almost any horospherical two-orbit $G$-variety, we can associate a horospherical $G$-variety of rank one, which is 
naturally embedded in a homogeneous $G\times \operatorname{GL}_{d+1}$-variety. Then we will prove that this variety is the zero locus of a general section of a vector bundle over this homogeneous space.

\subsection{The general construction}	
Our main construction will involve a certain fundamental $G$-module
 $V$, whose dimension will be denoted by $d$.
Let $k$ be a positive integer, with $k<d$. Denote $e_1,\dots,e_{k+1}$
$T$-semi-invariant linearly independent vectors of maximal weights in $V$. We order them starting from the highest weight, in some compatible way with the partial dominance order. In particular 
 $v_k:=e_1\wedge\dots\wedge e_k$ and $v_{k+1}:=e_1\wedge\dots\wedge e_{k+1}$ are $B$-semi-invariant vectors of $\bigwedge^kV$ and $\bigwedge^{k+1}V$, respectively. Denote by $V_k$ (respectively $V_{k+1}$) the sub-$G$-module of $\bigwedge^kV$ (resp. $\bigwedge^{k+1}V$) generated by $v_k$ (resp. $v_{k+1}$).

We will suppose that the weights of $v_k$ and $v_{k+1}$ are fundamental weights $\varpi_i$ and $\varpi_j$ of $(G,B,T)$, in particular $V_k\simeq V(\varpi_i)$ and $V_{k+1}\simeq V(\varpi_j)$.
This is a quite restrictive hypothesis, which is sensible only 
when $V$ is a fundamental module associated to an end of the Dynkin 
of diagram of $G$; in this situation the closed orbit in $\mathbb{P}(V_k)$ can naturally be realized as a subvariety of $G(k,V)$, as discussed 
in \cite{landsberg-manivel} (see in particular Proposition 4.15). 
Finally, fix a non-zero element $e_0$ in the trivial $G$-module $\Cbb$.

\begin{defi}With the notations above, we define $\tilde{X}$ as the $G$-orbit closure: $$\tilde{X}:=\overline{G\cdot[v_{k+1}\otimes(v_k\wedge (e_{k+1}+e_0))]}\subset \Pbb(V_{k+1}\otimes \bigwedge^{k+1}(V\oplus\Cbb)).$$
\end{defi}

Denote by $P_k$ and $P_{k+1}$ the maximal parabolic subgroups containing $B$ associated to $\varpi_i$ and $\varpi_j$ respectively. Note that $P_k$ and $P_{k+1}$ are the stabilizers in $G$ of the lines $\Cbb v_{k}$ and $\Cbb v_{k+1}$, respectively.

\begin{prop} The variety $\tilde{X}$ is a horospherical $G$-variety. Moreover, there exists a horospherical $G$-variety $X$ of Picard group $\Zbb$ such that $\tilde{X}$ is obtained by blowing-up $X$ 
along a closed $G$-orbit.
\end{prop}

To prove the proposition we will use the general theory of horospherical varieties, in particular, the classification in terms of colored fans, the description of divisors and the ampleness criterion. For a survey on this theory, see for example \cite{Fanohoro} or \cite{SingSpher}.

\begin{proof}
We start by studying the open $G$-orbit $\Omega$ of $\tilde{X}$ more closely. We have $\Omega\simeq G/H$ where $H=\operatorname{Stab}_G[v_{k+1}\otimes(v_k\wedge (e_{k+1}+e_0)]$ is the kernel of $\varpi_j-\varpi_i$ in the parabolic subgroup $P=P_k\cap P_{k+1}$ that stabilizes both $\Cbb v_{k}$ and $\Cbb v_{k+1}$.
In particular $\Omega$ is a horospherical homogeneous space of rank~1 with the following spherical data: 
\begin{enumerate}
\item the weight lattice $M=\Zbb (\varpi_{j}-\varpi_i)\simeq \Zbb$;
\item the set of colors $\mathcal{D}=\{D_k,D_{k+1}\}$ corresponding to the set of inverse images by $G/H\longrightarrow G/P$ of the two Schubert divisors in $G/P$;
\item the images of $D_k$ and $D_{k+1}$ in $N:=\Hom_\Zbb(M,\Zbb)\simeq \Zbb$, given respectively by $\alpha_{k|M}^\vee$ (i.e. $-1\in\Zbb$) and $\alpha_{k+1|M}^\vee$ (i.e. $1\in\Zbb$). 
\end{enumerate}  

According to the Luna-Vust classification of $G/H$-embeddings 
in terms of colored fans, there exist four complete $G/H$-embeddings, obtained by picking or not picking each of the two colors. Each of these embeddings  has three $G$-orbits: $\Omega$ and two closed orbits isomorphic to either $G/P_k$, $G/P_{k+1}$ or $G/P$ (which is a divisor).

The Picard group and the ample divisors of each of these $G/H$-embeddings can easily be described. In particular, they are all projective and locally factorial, and their Picard number can be $1$, $2$ or $3$. Also, in order to realize these embeddings, we can choose a (small) ample divisor (which is automatically very ample because $G/H$ has rank one \cite[Th.~0.3]{these}), and by computing its global sections, we can describe the corresponding embedding into the projective space of the dual $G$-module of global sections (here we
need to suppose that $G$ is simply connected, so that our line bundle can be $G$-linearized). In another point of view, the projective $G/H$-embeddings are classified and can be described by moment polytopes, see \cite[Section 2.3]{MMPhoro}.

For example, to get the $G/H$-embedding $X$ obtained by picking the two colors of $\mathcal{D}$, we can choose $D$ so that the corresponding moment polytope is $Q=[\varpi_i,\varpi_j]$, and then $$X=\overline{G\cdot[v_k+v_{k+1}]}\subset\Pbb(V_k\oplus V_{k+1}),$$ which is of Picard number one; and to get the $G/H$-embedding $X'$ obtained by picking only the color $D_k$, we can choose $D$ so that the corresponding moment polytope is $Q'=[\varpi_i+\varpi_j,2\varpi_j]$, and then $$X'=\overline{G\cdot[v_{k+1}\otimes v_k+v_{k+1}\otimes v_{k+1}]}\subset\Pbb((V_{k+1}\otimes V_k)\oplus (V_{k+1}\otimes V_{k+1})),$$ which is of Picard number two. 
To get the descritpion of $X'$, we also use that $V(\varpi_i+\varpi_j)\subset V(\varpi_j)\otimes V(\varpi_i)\simeq V_{k+1}\otimes V_k$ and $V(2\varpi_j)\subset V(\varpi_j)\otimes V(\varpi_j)\simeq V_{k+1}\otimes V_{k+1}$. We illustrate the choice of the moment polytope in Figure~\ref{fig:momentpolytopes}. 

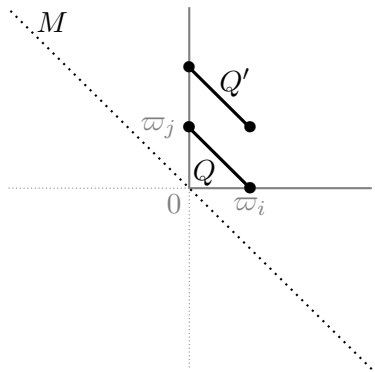
\begin{figure}

\begin{center}
\begin{tikzpicture}[scale=0.8]

\draw[thick,color=gray ] (0,0) -- (3,0);
\draw[thick,color=gray ] (0,0) -- (0,3);
\draw[densely dotted,color=gray] (0,0) -- (-3,0);
\draw[densely dotted,color=gray] (0,0) -- (0,-3);
\draw[thick,dotted] (3,-3) -- (-3,3);
\node at (-2.25,2.75) {$M$};
\node[color=gray ] at (-0.25,-0.25) {0};
\node[color=gray ] at (1,-0.3) {$\varpi_i$};
\node[color=gray ] at (-0.5,1) {$\varpi_j$};
\node[color=gray ] at (1,0) {$\bullet$};
\node[color=gray ] at (0,1) {$\bullet$};
\draw[very thick] (0,1) -- (1,0);
\node at (1,0) {$\bullet$};
\node at (0,1) {$\bullet$};
\node at (0.25,0.25) {$Q$};
\node at (0.75,1.75) {$Q'$};
\draw[very thick] (0,2) -- (1,1);
\node at (0,2) {$\bullet$};
\node at (1,1) {$\bullet$};
\end{tikzpicture}
\caption{Two choices of moment polytopes for $G/H$}\label{fig:momentpolytopes}
\end{center}
\end{figure}

Remark that $X'$ can also be obtained by  blowing-up $X$ along the closed $G$-orbit isomorphic to $G/P_k$. Indeed, blowing-up one of the closed $G$-orbit in $X$ we obtain another projective $G/H$-embedding where the closed $G$-orbit of $X$ has been replaced by a $G$-stable divisor, which has to be a closed $G$-orbit isomorphic to $G/P$. In terms of colored fans, we have deleted a color.
Now, observe that 
$$(V_{k+1}\otimes V_k)\oplus (V_{k+1}\otimes V_{k+1})=V_{k+1}\otimes(V_k\oplus V_{k+1})\subset V_{k+1}\otimes (\bigwedge^kV\oplus\bigwedge^{k+1}V)\simeq V_{k+1}\otimes\bigwedge^{k+1}(V\oplus\Cbb),$$ and that by this $G$-equivariant isomorphism, $$X'\simeq\overline{G\cdot[v_{k+1}\otimes(v_k\wedge(e_{k+1}+e_0))]}\subset\Pbb(V_{k+1}\otimes\bigwedge^{k+1}(V\oplus\Cbb)).$$ 
This exactly means that $X'$ is $G$-equivariantly isomorphic to $\tilde{X}$.
\end{proof}

\begin{rem}
With the same $G$-equivariant isomorphism as above, we also have the following $G$-equivariant emmbedding $$X\simeq\overline{G\cdot[v_k\wedge(e_{k+1}+e_0)]}\subset\Pbb(\bigwedge^{k+1}(V\oplus\Cbb)).$$
\end{rem}

Now, $\tilde{X}$ is by construction a closed subvariety of the
homogeneous $G\times\operatorname{GL}_{d+1}$-space 
$$G/P_{k+1}\times \mathbb{G}(k+1,V\oplus\Cbb)\subset  \Pbb(V_{k+1})\times\Pbb (\bigwedge^{k+1}(V\oplus\Cbb))\subset  \Pbb(V_{k+1}\otimes\bigwedge^{k+1}(V\oplus\Cbb)).$$ 
Moreover, the immersion $\tilde{X}\longrightarrow G/P_{k+1}\times \mathbb{G}(k+1,V\oplus\Cbb)$ is $G$-equivariant so that composing 
with the projection to the second factor 
we get a proper $G$-equivariant morphism $\pi:\,\tilde{X}\longrightarrow \mathbb{G}(k+1,V\oplus\Cbb)$. The image of $\pi$ is the closure of the image of $G\cdot [v_k\wedge (e_{k+1}+e_0)]$ in $\Pbb(\bigwedge^{k+1}(V\oplus\Cbb))$, which is nothing else than $X$. This means that the projection 
to $\mathbb{G}(k+1,V\oplus\Cbb)$ induces the blow-up $X'\simeq\tilde{X}\rightarrow X$.

Note also that the closed orbit $G\cdot[v_{k+1}\otimes(v_k\wedge e_{k+1})]\simeq G/P_{k+1}$ maps by $\pi$ to $Z:=G\cdot[v_k\wedge e_{k+1}]\simeq G/P_{k+1}$. Moreover $G\cdot[v_{k+1}\otimes(v_k\wedge e_0)]\simeq G/(P_k\cap P_{k+1})$, which is the exceptional divisor of the blow-up, maps by to $Y:=G\cdot[v_k\wedge e_0]\simeq G/P_k$.

\subsection{The variety $\tilde{X}$ as the zero locus of a general section}

We begin by the following description of $\tilde{X}$. Recall that 
$G/P_{k+1}$ is embedded in $\mathbb{G}(k+1,V)$.

\begin{prop}\label{prop:1}
	Let $(A,B)\in G/P_{k+1}\times \mathbb{G}(k+1,V\oplus\Cbb)$, in particular $A$ and $B$ are subspaces of dimension~$k+1$ of 
	$V$ and $V\oplus\Cbb$,  respectively. Then $(A,B)$ is a point of $\tilde{X}$ if and only if the projection of $B$ to $V$ is contained in $A$, or equivalently $B\subset A \oplus\Cbb $.
	\end{prop}
	
	\begin{proof}
	The points of the $G$-orbit $G\cdot[v_{k+1}\otimes (v_k\wedge (e_{k+1}+e_0))]$ are of the form $[(x_1\wedge\dots\wedge x_{k+1})\otimes (x_1\wedge\dots\wedge (x_{k+1}+e_0))]$, and then correspond to a pair $(A,B)\in G/P_{k+1}\times \mathbb{G}(k+1,V\oplus\Cbb)$ where $A$ is generated by $x_1,\dots,x_{k+1}$, and $B$ is generated by $x_1,\dots,x_k$ and $x_{k+1}+e_0$. In particular, the projection of $B$ in $V$ is contained in $A$. Since this is a closed condition, it remains  true on the closure of $G\cdot[v_{k+1}\otimes (v_k\wedge (e_{k+1}+e_0))]$. 
	
	Conversely, if the projection of $B$ in $V$ is contained in $A$, there exists a basis $(x_1,\dots,x_{k+1})$ of $A$ such that $(x_1,\dots,x_k,x_{k+1}+e_0)$ or $(x_1,\dots,x_{k+1})$ or $(x_1,\dots,x_k,e_0)$ is a basis of $B$. In the first case, $(A,B)$ is a point of the open $G$-orbit $G\cdot[v_{k+1}\otimes (v_k\wedge (e_{k+1}+e_0))]$. In the two other cases, $(A,B)$ is 
	a point in one of the two closed $G$-orbits $G\cdot[v_{k+1}\otimes v_{k+1}]$ and $G\cdot[v_{k+1}\otimes (v_k\wedge e_0)]$. 
	\end{proof}
	
	We will deduce 
	the following proposition, which completes the proof of Theorem~\ref{th:main}. 
	Since $G/P_{k+1}$ is embedded in the Grassmannian $\mathbb{G}(k+1,V)$, it admits a vector bundle $\mathcal{Q}$
	obtained by restricting the  tautological quotient vector bundle, of rank $d-k-1$.
	We also denote by 
	$\mathcal{U}$  the tautological vector bundle, of rank $k+1$,  over $\mathbb{G}(k+1,V\oplus\Cbb)$.
	
	\begin{prop}\label{prop:2}
	The variety $\tilde{X}$ is the zero locus of a general section of the vector bundle $\mathcal{E}=\mathcal{Q}\boxtimes\mathcal{U}^*$ over $G/P_{k+1}\times \mathbb{G}(k+1,V\oplus\Cbb)$.
	\end{prop}
	
	We first prove the following statement.
	
	\begin{lem} The space of global sections of $\mathcal{E}$ contains
	$V\otimes (V\oplus\Cbb)^*.$
	\end{lem}

	\begin{proof}[Proof of the lemma]
	By the Borel-Weil theorem on the Grassmannian, $$H^0(\mathbb{G}(k+1,V\oplus\Cbb), \mathcal{U}^*)= (V\oplus\Cbb)^*, $$
	so it suffices to prove that $H^0(G/P_{k+1}, \mathcal{Q})$ contains $V$. But by its very definition, $Q$ is a quotient 
	of the trivial bundle with fiber $V$, so there is a non trivial equivariant map $V\rightarrow H^0(G/P_{k+1}, \mathcal{Q})$.
	Since $V$ is irreducible, it must be injective. 
	\end{proof}
	
	We can then conclude with the following general statement. 
	Suppose given a globally generated rank $r$ vector bundle $\mathcal{F}$ over some variety $Z$, and denote by $F$ its space of global sections, of dimension $d$. Consider a vector space $W$ of dimension $w\ge d$, some positive integer $\ell<w$, and denote
	by $\mathcal{U}$ the tautological rank $\ell$ bundle on the 
	Grassmannian $\mathbb{G}(\ell,W)$. The space of global sections 
	of $\mathcal{F}\boxtimes \mathcal{U}^*$ on $Z\times \mathbb{G}(\ell,W)$ is then 
	$$H^0(Z\times \mathbb{G}(\ell,W), \mathcal{F}\boxtimes \mathcal{U}^*)=F\otimes W^*\simeq Hom(W,F).$$ 
	Since $\mathcal{F}$ is globally generated, there is an exact sequence 
	$$0\longrightarrow \mathcal{M}\longrightarrow  F\otimes\mathcal{O}_Z
	\longrightarrow  \mathcal{F}\longrightarrow 0,$$
	where $\mathcal{M}$ is a vector bundle of rank $d-r$.
	
	\begin{prop}\label{gen}
	Let $s$ be a section of  $\mathcal{F}\boxtimes \mathcal{U}^*$,
	defined by $\sigma\in Hom(W,F)$. Then:
	\begin{enumerate}
	\item the zero locus $Z(s)\subset Z\times \mathbb{G}(\ell,W)$ 
	is the set of pairs $(x,U)$ such that $$\sigma(U)\subset  \mathcal{M}_x\subset F;$$ 
	\item if $\sigma$ is surjective, $Z(s)$ is smooth and has the 
	structure of a $\mathbb{G}(\ell,d-r)$-bundle over $Z$. More precisely, if $K\subset W$ denotes the kernel of $\sigma$, 
	$$Z(s)\simeq \mathbb{G}(\ell,K\otimes\mathcal{O}_Z\oplus \mathcal{M}).$$
	\end{enumerate}
	\end{prop}
	
	\begin{proof} 
	The first claim is clear, since the evaluation of $s$ at $(x,U)$ is the vector in $Hom(U,\mathcal{F}_x)$ obtained by restricting
	$\sigma$ to $U\subset W$ and projecting from $F$ to its quotient 
	$\mathcal{F}_x$. 
	
	The condition can be rewritten $U\subset \sigma^{-1}(\mathcal{M}_x)$. When $\sigma$ is surjective, 
	$\sigma^{-1}(\mathcal{M}_x)$ has constant dimension $e-r$. 
	This implies that the projection to $Z$ is locally trivial,
	and the second claim follows since $\sigma^{-1}(\mathcal{M}_x)$
	can be identified with $K\oplus \mathcal{M}_x$.
	\end{proof}

	\begin{proof}[Conclusion of the proof of Proposition~\ref{prop:2}] 
	Apply Proposition \ref{gen} (1) to the canonical section of 
	$\mathcal{E}$ given by $\sigma\in Hom(V\oplus \Cbb, V)$,
	the projection to $V$. By Proposition \ref{prop:2}, its zero locus coincides with $\tilde{X}$. \end{proof}
	
	\medskip By Proposition \ref{gen} (2), the projection of $\tilde{X}$ to $G/P_{k+1}$ is the fiber-bundle $\mathbb{G}(k+1,\Cbb\oplus\mathcal{V})$, if $\mathcal{V}$
	denotes the tautological bundle, restricted 
	from $\mathbb{G}(k+1,V)$. Since $\Cbb\oplus\mathcal{V}$ has rank $k+2$, this is just the hyperplane bundle $\mathbb{P}(\Cbb\oplus\mathcal{V}^*)$, of relative dimension $k+1$. This concludes the proof of the main Theorem.

	\begin{rem}
 The stabilizer of $\sigma$ in $G\times GL_{d+1}$ is  $(G\times \Cbb^*)\ltimes  V$. Indeed, $(g_1,g_2)\in G\times \operatorname{GL}_{d+1}$ acts on $Hom(V\oplus \Cbb, V)\simeq M_{d,d+1}(\Cbb)$ by $(g_1,g_2)\cdot M=g_1 M g_2^{-1}$; then $(g_1,g_2)\cdot (I_{d},\; 0)=(I_{d},\; 0)$ if and only if $$g_2=\left(\begin{array}{cc}
g_1 & 0\\
\!^tx  & y
\end{array}\right),$$ with $x\in V$ and $y\in\Cbb^*$.
As a consequence, the $G\times \operatorname{GL}_{d+1}$-orbit 
of  $\sigma$ is always dense, and the zero-loci of the 
corresponding  sections are all isomorphic to $\tilde{X}$.
\end{rem}

\subsection{More details on the module $V$}\label{sec:1.3}
Here discuss what $k$ and the module $V$ are for each case 
of Theorem \ref{2orbits} (except the $Spin_7$-variety).\medskip

\noindent {\bf $\bullet$ Case 1.}  $(\operatorname{Spin}_{2m+1}, P(\varpi_{m-1}), P(\varpi_m))$ with $m\geq 3$.
 
 \smallskip
The $G$-module $V$ is the spinorial representation, of dimension $d=2^m$, and $k=1$. The $T$-weights of $V$ are the $\frac{1}{2}(\pm\epsilon_1+\cdots+\pm\epsilon_m)$, all with multiplicity one. Let $e_1$ be a $T$-semi-invariant vector of weight $\varpi_m=\frac{1}{2}(\epsilon_1+\cdots+\epsilon_m)$ and let $e_2$ be a $T$-semi-invariant vector of weight $\varpi_m-\alpha_m=\frac{1}{2}(\epsilon_1+\cdots+\epsilon_{m-1}-\epsilon_m)$. Then $e_1\wedge e_2$ is of weight $\varpi_{m-1}=\epsilon_1+\cdots+\epsilon_{m-1}$.\\

\noindent {\bf $\bullet$ Case 3.} $(\operatorname{Sp}_{2m},P(\varpi_{k+1}),P(\varpi_k))$ with $m\geq 2$.

\smallskip
The $G$-module $V$ is the minimal representation, of dimension $d=2m$,  and $k$ is any positive  integer smaller than $m$.
The $T$-weights of $V$ are the $\pm\epsilon_h$ with $1\leq h\leq m$, all of multiplicity one. For any $h\in\{1,\dots,k+1\}$, let $e_h$ be a $T$-semi-invariant vector of weight $\varpi_1-\alpha_1-\cdots-\alpha_{h-1}=\epsilon_h$. Then $e_1\wedge\dots\wedge e_k$ and $e_1\wedge\dots\wedge e_{k+1}$ are of weights $\varpi_k=\epsilon_1+\cdots+\epsilon_k$ and $\varpi_{k+1}=\epsilon_1+\cdots+\epsilon_{k+1}$ respectively.\\
  
\noindent {\bf  $\bullet$ Case 4.}  $(F_4,P(\varpi_2),P(\varpi_3))$.
  
  \smallskip
The $G$-module $V$ is the minimal representation, of dimension $d=26$,  and $k=2$.
The highest weight of $V$ is $\varpi_4=\epsilon_1$. Consider 
$T$-semi-invariant vectors 
$e_1$  of weight $\varpi_4$,  $e_2$  of weight $\varpi_4-\alpha_4=\frac{1}{2}(\epsilon_1+\epsilon_2+\epsilon_3+\epsilon_4)$, and $e_3$ of weight $\varpi_4-\alpha_4-\alpha_3=\frac{1}{2}(\epsilon_1+\epsilon_2+\epsilon_3-\epsilon_4)$. Then $e_1\wedge e_2$ is of weight $\varpi_3=\frac{1}{2}(3\epsilon_1+\epsilon_2+\epsilon_3+\epsilon_4)$,
while $e_1\wedge e_2\wedge e_3$ is of weight $\varpi_2=2\epsilon_1+\epsilon_2+\epsilon_3$.

\noindent {\bf   $\bullet$ Case 5.}  $(G_2,P(\varpi_2),P(\varpi_1))$.
  
  \smallskip
The $G$-module $V$ is the minimal representation,  of dimension $d=7$,  and $k=1$.
The highest weight of $V$ is $\varpi_1$.  Consider 
$T$-semi-invariant vectors  $e_1$  of weight $\varpi_1$,
and  $e_2$  of weight $\varpi_1-\alpha_1$. Then $e_1\wedge e_2$ is of weight $2\varpi_1-\alpha_1=\varpi_2$.\\

\section{Cohomology of the $G_2$-variety}

\subsection{The Hasse diagram}
Recall that the Chow ring of the horospherical variety $X$
of type $G_2$ has two natural basis, made of classes coming 
from the two closed $G_2$-orbits \cite{GPPS}. The latter are $Z=G_2/P_1\simeq\Qbb^5$, the only closed $Aut(X)$-orbit,  and 
the adjoint variety $Y=G_2/P_2$, also of dimension $5$. 
Both are homologically rational
projective homogeneous spaces, in the sense that their Hodge numbers are
the same as those of $\Pbb^5$. We deduce that the Chow ring of
$X$ is free of rank $6+6=12$. 

Let us choose the basis $(\tau'_i,\sigma_j)$, made of classes
indexed by their degrees, where the $\tau'_i$ are induced from 
$Z=G_2/P_1$ in degree $0$ to $5$, while the $\sigma_j$ are induced from $Y=G_2/P_2$ in degree $2$ to $7$. More precisely, the $\tau'_i$ are the classes of the closures in $X$ of $\Cbb^2$-bundles over Schubert varieties of $G_2/P_1$ (locally $X$ is a $\Cbb^2$-bundle over $G_2/P_1$), and the $\sigma_j$ are the classes of the Schubert varieties of $G_2/P_2$.

From Proposition 1.14 of \cite{GPPS} we deduce that the Hasse diagram of $X$ is the 
following:

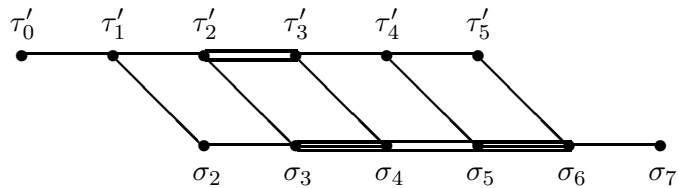
\setlength{\unitlength}{6mm}
\thicklines
\begin{picture}(20,5.3)(-4.5,-1.5)
\multiput(0,2)(2,0){6}{$\bullet$}
\multiput(4,0)(2,0){6}{$\bullet$}

\multiput(0.2,2.2)(2,0){2}{\line(1,0){2}}
\multiput(6.2,2.2)(2,0){2}{\line(1,0){2}}
\put(4.2,2.06){\line(1,0){2}}\put(4.2,2.26){\line(1,0){2}}
\put(4.2,.2){\line(1,0){2}}\put(12.2,.2){\line(1,0){2}}
\multiput(6.2,.05)(0,.1){3}{\line(1,0){2}}
\multiput(10.2,.05)(0,.1){3}{\line(1,0){2}}
\put(8.2,.06){\line(1,0){2}}\put(8.2,.26){\line(1,0){2}}
\multiput(2.25,2.05)(2,0){5}{\line(1,-1){1.9}}

\put(-.1,2.7){$\tau'_0$}\put(1.9,2.7){$\tau'_1$}
\put(3.9,2.7){$\tau'_2$}\put(5.9,2.7){$\tau'_3$}
\put(7.9,2.7){$\tau'_4$}\put(9.9,2.7){$\tau'_5$}

\put(3.9,-.6){$\sigma_2$}\put(5.9,-.6){$\sigma_3$}
\put(7.9,-.6){$\sigma_4$}\put(9.9,-.6){$\sigma_5$}
\put(11.9,-.6){$\sigma_6$}\put(13.9,-.6){$\sigma_7$}
\end{picture}

Recall that the edges in this diagram encode the multiplication by the 
hyperplane class $h$ (which is nothing else than $\tau'_1$). For example 
$h\tau'_2=2\tau'_3+\sigma_3$. In particular we can readily deduce 
the 
degrees of the classes $(\tau'_i,\sigma_j)$, which are given in the following
diagram:

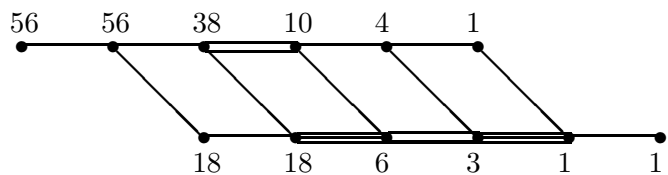
\setlength{\unitlength}{6mm}
\thicklines
\begin{picture}(20,5.3)(-4.5,-1.5)
\multiput(0,2)(2,0){6}{$\bullet$}
\multiput(4,0)(2,0){6}{$\bullet$}

\multiput(0.2,2.2)(2,0){2}{\line(1,0){2}}
\multiput(6.2,2.2)(2,0){2}{\line(1,0){2}}
\put(4.2,2.06){\line(1,0){2}}\put(4.2,2.26){\line(1,0){2}}
\put(4.2,.2){\line(1,0){2}}\put(12.2,.2){\line(1,0){2}}
\multiput(6.2,.05)(0,.1){3}{\line(1,0){2}}
\multiput(10.2,.05)(0,.1){3}{\line(1,0){2}}
\put(8.2,.06){\line(1,0){2}}\put(8.2,.26){\line(1,0){2}}
\multiput(2.25,2.05)(2,0){5}{\line(1,-1){1.9}}

\put(-.1,2.5){$56$}\put(1.9,2.5){$56$}
\put(3.9,2.5){$38$}\put(5.9,2.5){$10$}
\put(7.9,2.5){$4$}\put(9.9,2.5){$1$}

\put(3.9,-.6){$18$}\put(5.9,-.6){$18$}
\put(7.9,-.6){$6$}\put(9.9,-.6){$3$}
\put(11.9,-.6){$1$}\put(13.9,-.6){$1$}
\end{picture}

\subsection{Fundamental class}
According to Theorem \ref{th:main},  the blow-up $\tilde X$ of $X$ along its closed orbit $Z$ is the zero-locus of a general section of the vector bundle $\mathcal{Q}\boxtimes \mathcal{U}^*$ over the product variety $G_2/P_2\times \Gbb$, 
where $\Gbb:=G(2,V_7\oplus \Cbb)$.
The Thom-Porteous formula implies that its fundamental class
$$[\tilde X]=c_{10}(\mathcal{Q}\boxtimes \mathcal{U}^*)\in A^{10}(G_2/P_2\times 
\Gbb).$$
We can easily deduce the fundamental class of $X\subset \Gbb$.
We will denote by $\bar\sigma_{ij}$ the Schubert classes on $\Gbb$,
for $0\le j\le i\le 6$, and by $\bar\tau_{k\ell}$, for $0\le \ell\le k\le 5$, the Schubert classes on $G(2,V_7)$.

\begin{lem}
The fundamental class of $X\subset \Gbb$ is 
$$[X]=2\bar\sigma_{41}+2\bar\sigma_{32}\in A^5(\Gbb) .$$
\end{lem}

\proof Decompose $[\tilde X]=\sum_k\sum_i\alpha^i_k\otimes\beta^i_{10-k}$,
with $\alpha^i_k\in A^k(G_2/P_2)$ and $\beta^j_\ell\in A^\ell(\Gbb)$. Since the projection map $p$ from $\tilde X$ to $\Gbb$ is birational on its image $X$, we deduce that 
$$[X]=p_*[\tilde X]=\sum_i (p_*\alpha^i_5)\beta_5^i.$$
In order to compute this, let us denote by $x_1,\ldots ,x_5$ the Chern roots of $\mathcal{Q}$, 
and by $y_1,y_2$ the Chern roots of $\mathcal{U}^*$. 
Recall that the $m$-th elementary symmetric function of $x_1,\ldots ,x_5$ is the $m$-th Chern class of $\mathcal{Q}$,
which is nothing else than the Schubert class $\bar\tau_m$ of $G(2,V_7)$,
restricted to $G_2/P_2$. We get
$$c_{10}(\mathcal{Q}\boxtimes \mathcal{U}^*)=\prod_{i,j}(y_i+x_j)=\prod_{i=1}^2(y_i^5+
y_i^4\bar\tau_1+y_i^3\bar\tau_2+y_i^2\bar\tau_3+y_i\bar\tau_4+\bar\tau_5).$$
The part of bidegree $(5,5)$ is 
$$[\tilde X]_{5,5}=(y_1^5+y_2^5)\bar\tau_5+(y_1^4y_2+y_1
y_2^4)\bar\tau_4\bar\tau_1+(y_1^3y_2^2+y_1^2y_2^3)\bar\tau_3\bar\tau_2.$$
In order to project this, we need to evaluate the classes 
$\bar\tau_5$, $\bar\tau_4\bar\tau_1$ and $\bar\tau_3\bar\tau_2$ on $G_2/P_2$.
For this we need to recall that $G_2/P_2\subset G(2,V_7)$ is the 
zero locus of a general section of the vector bundle $\mathcal{Q}^*(1)$. 
In particular, its fundamental class is 
$$[G_2/P_2]=c_5(\mathcal{Q}^*(1))=2\bar\tau_{41}+2\bar\tau_{32}\in A^5(G(2,V_7)).$$
For any class $\alpha$ restricted from the Grassmannian, it is
then straightforward to compute 
$$\int_{G_2/P_2}\alpha =\int_{G(2,V_7)}(2\bar\tau_{41}+2\bar\tau_{32})\alpha.$$
In particular, we get the following evaluations:
$$\int_{G_2/P_2}\bar\tau_5=0, \qquad \int_{G_2/P_2}\bar\tau_4\bar\tau_1=2,
\qquad \int_{G_2/P_2}\bar\tau_3\bar\tau_2=4.$$
Plugging in our formula for the fundamental class of $X$, we finally get
$$ [X]=2(y_1^4y_2+y_1y_2^4)+4(y_1^3y_2^2+y_1^2y_2^3)=2\bar\sigma_2(\bar\sigma_1^3-\bar\sigma_1\bar\sigma_2)=2\bar\sigma_{41}+2\bar\sigma_{32}.\qed $$

\subsection{Generators and relations}
Our next ingredient in order to compute the intersection product on $X$ is

\begin{lem} The restriction map $A^*(\Gbb)_\Qbb\rightarrow A^*(X)_\Qbb$ is surjective. And $\tau'_1$ and $\sigma_2$ are the restrictions of  the Schubert classes $\bar\sigma_1$ and $\bar\sigma_{11}$ of $\Gbb$, respectively. \end{lem}

\begin{proof}
The ring $A^*(X)_\Qbb$ is generated by $\tau'_1$ and $\sigma_2$. It is therefore enough to prove that these two classes  
are the restrictions of the Schubert classes $\bar\sigma_1$ and $\bar\sigma_{11}$ of $\Gbb$,  respectively. 
Remark that $\tau'_1$ is the hyperplane class in $X$ while $\bar\sigma_1$ is the hyperplane class of $G$, so the claim 
is obvious for $\tau'_1$.

Recall that $\sigma_2$ is the class of the closed $G_2$-orbit $Y$ of $X$, which is isomorphic to $G_2/P_2$. In a neighborhood of $Y$, more precisely on $X\backslash Z$, $X$ is a vector bundle of rank two over $Y$. More precisely, the map $X\backslash Z \longrightarrow Y$ is given by the restriction of the projection map $\Gbb=\mathbb{G}(2,V_7\oplus\Cbb)\dashrightarrow \mathbb{G}(2,V_7)$. This projection is defined on $\Gbb\backslash \{W\in \Gbb,\,\mid\,\Cbb\subset W\}$, which is a neighborhood of the Schubert variety $\Gbb_{11}:=\{W\in \Gbb\,\mid\, W\subset V_7\}$; and it defines a vector bundle of rank two over $\mathbb{G}(2,V_7)$. It is now clear that the restriction of $\bar\sigma_{11}$ (the class of $\Gbb_{11}$) is $\sigma_2$.
\end{proof}

\begin{rem}
The ring $A^*(X)_\Qbb$ is also generated by $\tau'_1$ and $\tau'_2$. We could also prove directly 
that $\tau'_2$ is the restriction of the Schubert class $\bar\sigma_2$, the class of $\{W\in \Gbb\,\mid\, W\cap V_5\neq \{0\}\}$, where $V_5$ is a 5-dimensional subspace of $V_7$. Indeed, one can check that this Schubert variety intersects $X$ transversely at general points, and that the intersection is the subvariety of $X$ that defines $\tau'_2$.
\end{rem}

In order to simplify the notations we will denote by $h, \sigma$ 
our two generators  $\tau'_1, \sigma_2$ of the Chow ring of $X$.
Given the Betti numbers of $X$, we deduce that the rational Chow ring
$$A^*(X)_\Qbb = \Qbb[h,\sigma]/\langle R_4,R_6\rangle $$
for two relations $R_4$ of degree four and $R_6$ of degree six. 

\begin{prop}\label{rel}
We can choose the relations to be 
$$R_4=3\sigma^2-h^2\sigma \qquad \mathit{and}\qquad R_6=28h^4\sigma-9h^6.$$
\end{prop}

\proof The fact that $R_4=0$ in $A^*(X)$ follows from the observation that in the Chow ring of $G(2,V_7)$, we have
$[X]\bar\sigma_{22}=2\bar\sigma_{54}$, while $[X]\bar\sigma_{31}=4\bar\sigma_{54}$. Therefore the class $2\bar\sigma_{22}
-\bar\sigma_{31}=3\bar\sigma_{11}^2-\bar\sigma_1^2\bar\sigma_{11}$ restricts to zero on $X$, and this restriction is $3\sigma^2-h^2\sigma$. 

The fact that $R_6=0$ is even easier. Indeed $A^6(X)$ has rank
one, so the classes $h^6$ and $h^4\sigma$ must be proportional. 
Since the degree of the latter is $18$, while the degree of the
former is $56$, the claim follows immediately. \qed

\subsection{The multiplication table}
Note that the information encoded in the Hasse diagram 
is already sufficient to express all the classes of $X$ in terms of $h$ and $\sigma$. We get: 
$$\tau'_2=h^2-\sigma, \quad \tau'_3=\frac{h^3}{2}-h\sigma, \quad \tau'_4=\frac{h^4}{2}-\frac{4}{3}h^2\sigma, \quad 
 \tau'_5=\frac{h^5}{2}-\frac{3}{2}h^3\sigma, $$
 $$\sigma_3=h\sigma, \quad \sigma_4=\frac{h^2\sigma}{3}, \quad
\sigma_5=\frac{h^3\sigma}{6}, \quad \sigma_6=\frac{h^6}{56}, \quad
\sigma_7=\frac{h^7}{56}.$$

\medskip
Using the relations of Proposition \ref{rel}, the multiplication 
table is then easily obtained: 

$$(\tau'_2)^2=2\tau'_4+3\sigma_4, \qquad \tau'_2\sigma_2=2\sigma_4, \qquad \sigma_2^2=\sigma_4,$$
$$\tau'_3\sigma_2=\sigma_5, \quad \sigma_3\sigma_2=2\sigma_5,
\qquad \tau'_3\tau'_2=\tau'_5+2\sigma_5, \quad \sigma_3\tau'_2=4\sigma_5,$$
$$\tau'_4\sigma_2=\sigma_6, \quad \sigma_4\sigma_2=2\sigma_6,
\qquad \tau'_4\tau'_2=3\sigma_6, \quad \sigma_4\tau'_2=4\sigma_6,$$
$$\tau'_5\sigma_2=0, \quad \sigma_5\sigma_2=\sigma_7,
\qquad \tau'_5\tau'_2=\sigma_7, \quad \sigma_5\tau'_2=2\sigma_7,$$
$$(\tau'_3)^2=2\tau'_6, \qquad \tau'_3\sigma_3=3\sigma_6, \qquad \sigma_3^2=6\sigma_6,$$
$$\tau'_4\sigma_3=\sigma_7, \quad \sigma_4\sigma_3=2\sigma_7,
\qquad \tau'_4\tau'_3=\sigma_7, \quad \sigma_4\tau'_3=\sigma_7.$$

\medskip
For completeness, we can also compute the Poincar\'e dual basis
(which is not, as in classical Schubert calculus, a permutation 
of the original basis). Following the notations of \cite[Proposition 1.10]{GPPS}, we let $\sigma'_i=\sigma_{7-i}^\vee$ and $\tau_j=(\tau'_{7-j})^\vee$. Then:

$$\tau_6=\sigma_6, \quad \tau_5=\tau'_5, \quad 
\sigma'_5=\sigma_5-2\tau'_5, \quad  \tau_4=2\tau'_4-\sigma_4, \quad 
\sigma'_4=\sigma_4-\tau'_4,$$
$$\tau_3=2\tau'_3-\sigma_3, \quad 
\sigma'_3=\sigma_3-\tau'_3, \quad \tau_2=\tau'_2-2\sigma_2, \quad 
\sigma'_2=\sigma_2, \quad \sigma'_1 = \tau'_1.$$

\medskip
This is the other natural basis of the Chow ring of $X$, in terms of which the Hasse diagram becomes, in agreement with 
\cite[Proposition 4.6]{GPPS}:

\setlength{\unitlength}{6mm}
\thicklines
\begin{picture}(20,5.6)(-4.5,-1.5)
\multiput(0,2)(2,0){6}{$\bullet$}
\multiput(4,0)(2,0){6}{$\bullet$}

\multiput(4.2,0.2)(2,0){2}{\line(1,0){2}}
\multiput(10.2,0.2)(2,0){2}{\line(1,0){2}}
\put(8.2,0.06){\line(1,0){2}}\put(8.2,0.26){\line(1,0){2}}

\put(0.2,2.2){\line(1,0){2}}\put(8.2,2.2){\line(1,0){2}}
\multiput(2.2,2.05)(0,.1){3}{\line(1,0){2}}
\multiput(6.2,2.05)(0,.1){3}{\line(1,0){2}}
\put(4.2,2.06){\line(1,0){2}}\put(4.2,2.26){\line(1,0){2}}

\multiput(2.25,2.05)(2,0){5}{\line(1,-1){1.9}}

\put(-.1,2.7){$\sigma'_0$}\put(1.9,2.7){$\sigma'_1$}
\put(3.9,2.7){$\sigma'_2$}\put(5.9,2.7){$\sigma'_3$}
\put(7.9,2.7){$\sigma'_4$}\put(9.9,2.7){$\sigma'_5$}

\put(3.9,-.6){$\tau_2$}\put(5.9,-.6){$\tau_3$}
\put(7.9,-.6){$\tau_4$}\put(9.9,-.6){$\tau_5$}
\put(11.9,-.6){$\tau_6$}\put(13.9,-.6){$\tau_7$}
\end{picture}

In this basis, the multiplication table is the following:

$$(\sigma'_2)^2=2\sigma'_4, \qquad \sigma'_2\tau_2=0, \qquad \tau_2^2=\tau_4,$$
$$\sigma'_3\tau_2=-\tau_5, \quad \tau_3\tau_2=2\tau_5,
\qquad \sigma'_3\sigma'_2=\sigma'_5+2\tau_5, \quad \tau_3\sigma'_2=0,$$
$$\sigma'_4\tau_2=-\tau_6, \quad \tau_4\tau_2=2\tau_6,
\qquad \sigma'_4\sigma'_2=\tau_6, \quad \tau_4\sigma'_2=0,$$
$$\sigma'_5\tau_2=-2\tau_7, \quad \tau_5\tau_2=\tau_7,
\qquad \sigma'_5\sigma'_2=\tau_7, \quad \tau_5\sigma'_2=0,$$
$$(\sigma'_3)^2=2\tau_6, \qquad \sigma'_3\tau_3=-\tau_6, \qquad \tau_3^2=2\tau_6,$$
$$\sigma'_4\tau_3=-\tau_7, \quad \tau_4\tau_3=2\tau_7,
\qquad \sigma'_4\sigma'_3=\tau_7, \quad \tau_4\sigma'_3=-\tau_7.$$

\medskip\noindent {\it Remarks.}
\begin{enumerate} 
\item One important difference between the two multiplication tables is that the second one has some negative signs, while the 
first one has none. This is due to the fact that $\tau_2$ is the 
class of $G_2/P_1\subset X$, which is the closed $Aut(X)$-orbit $Z$
in $X$ and is therefore not movable. On the contrary, $\sigma$ is the class of a restricted Schubert cycle, and is 
therefore movable in $X$.
\item The degree four relation $R_4$ can be expressed as 
$\sigma_2\tau_2=0$, and obviously follows from the fact that the 
two closed $G$-orbits of $X$ do not meet. 
\item More generally, and for any horospherical variety $X$ with 
Picard number one, we have inside $A^*(X)$ two subalgebras 
$A_1$ (here generated by the $\sigma_i$'s) and $A_2$ (here generated by the $\tau_j$'s) such that 
$$hA_1\subset A_1, \qquad hA_2\subset A_2, \qquad A_1A_2=0.$$
Formally we can even decompose $A^*(X)=A_1\oplus A_2^\vee=
A_2\oplus A_1^\vee$ (where $A_1^\vee$ is the submodule generated by the Poincar\'e duals of the Schubert classes in $A_1$). 
Then 
$$A_1^\vee A_2\subset A_2 \qquad \mathrm{and}\qquad A_2^\vee A_1\subset A_1.$$
If we add Poincar\'e duality and the Chevalley formula, do we get enough information to determine $A^*(X)$?
\end{enumerate}

\subsection{Quantum cohomology}

Recall that $X$ has index four, so that the quantum parameter in its quantum cohomology ring $QA^*(X)=A^*(X)_\Qbb[q]$ 
has degree 
four. By the general results of Siebert and Tian \cite{st}, 
this quantum cohomology ring admits a presentation of the form
$$QA^*(X)_\Qbb = \Qbb[h,\sigma,q]/\langle R_4(q),R_6(q)\rangle $$
for two relations $R_4(q)$ of degree four and $R_6(q)$ of degree six, which are $q$-deformations of $R_4$ and $R_6$. In particular
we can write 
$$R_4(q)=R_4+r_4 q \qquad \mathrm{and} \qquad R_6(q)=R_6+r_6 q,$$
where $r_4$ has degree zero (a rational number) and $r_6$ is a class of degree two. Note that since there is no term of degree bigger that one in $q$, these relations are determined by degree one Gromov-Witten invariants only.

The quantum Chevalley formula has been computed in Proposition 4.6 of \cite{GPPS}, in terms of our Poincar\'e dual basis. In the basis $(\sigma_i,\tau'_j)$, this reads 

$$h*h=\sigma_2+\tau'_2, \quad \tau'_2*h=2\tau'_3+\sigma_3, \quad 
\sigma_2*h=\sigma_3,$$
$$\tau'_3*h=\tau'_4+\sigma_4+q, \qquad 
\sigma_3*h=3\sigma_4+q,$$
$$\tau'_4*h=\tau'_5+\sigma_5+qh, \qquad 
\sigma_4*h=2\sigma_5+qh,$$
$$\tau'_5*h=\sigma_6+q\sigma_2, \qquad 
\sigma_5*h=3\sigma_6+q\tau'_2,$$
$$\sigma_6*h=\sigma_7+q\tau'_3, \qquad 
\sigma_7*h=q\tau'_4+2q^2.$$
\smallskip 

One immediately deduces the quantum Giambelli formulas (that is, how to determine 
each Schubert class in terms of the generators), and also the 
relation $R_6(q)$. Indeed, we get in quantum cohomology 

$$\tau'_2*h^4=18\tau'_6+q(10\sigma_2+4\tau'_2), \qquad 
h^6=56\tau'_6+16q(2\sigma_2+\tau'_2),$$
\smallskip 

\noindent and therefore $R_6(q)=R_6+8q(h^2+3\sigma)$. 
So the only ingredient missing is the computation of $\sigma^2$ in quantum cohomology, which is given 
by $$\sigma^2=\sigma_4+I_1(\sigma,\sigma,\sigma_7)q.$$

\begin{lem}
The Gromov-Witten invariant $ I_1(\sigma,\sigma,\sigma_7)=0$.
\end{lem}

\begin{proof}
Since the class $\tau$ is movable, the Gromov-Witten invariant 
$I_1(\sigma,\sigma,\sigma_7)$ is enumerative \cite[Section 3.2]{GPPS}: it counts the number of
lines $\ell$ in $X$ that pass through a general point $x$ (representing a plane $P_x$ in $V_7\oplus \Cbb$), and meet 
general Schubert cycles of class $\sigma_{11}$, that is, two 
general sub-Grassmannians $G(2,A_7)$ and $G(2,B_7)$, for 
$A_7,B_7$ two general hyperplanes of $V_7\oplus \Cbb$. 
But a projective line $\ell$ in $G(2,V_7\oplus \Cbb)$ is made of planes containing a common line $L_1$ (and contained in a common
three dimensional space $L_3$). We would thus  get the inclusion $L_1\subset P_x\cap A_7\cap B_7=0$, a contradiction.
\end{proof}

Since there is no other quantum correction, we deduce:

\begin{prop} The quantum cohomology ring of $X$ is 
$$QA^*(X)=\Qbb[h,\sigma,q]/\langle 3\sigma^2-h^2\sigma+q, 28h^4\sigma-9h^6+8q(h^2+3\sigma)\rangle .$$
\end{prop}

\medskip
Using the quantum Giambelli formulas
$$\tau'_2=h^2-\sigma, \quad \tau'_3=\frac{h^3}{2}-h\sigma, \quad \tau'_4=\frac{h^4}{2}-\frac{4}{3}h^2\sigma-\frac{2}{3}q, \quad 
 \tau'_5=\frac{h^5}{2}-\frac{3}{2}h^3\sigma-qh, $$
 $$\sigma_3=h\sigma, \qquad \sigma_4=\frac{h^2\sigma}{3}-\frac{1}{3}q, \qquad
\sigma_5=\frac{h^3\sigma}{6}-\frac{2}{3}qh, $$ $$\sigma_6=\frac{h^6}{56}+\frac{2}{7}q\sigma-\frac{4}{7}qh^2, \qquad
\sigma_7=\frac{h^7}{56}+\frac{9}{7}qh\sigma-\frac{15}{14}qh^3,$$ 
\smallskip

\noindent 
it would then be easy to deduce the quantum multiplication table in our basis (or the dual one). Let us just mention that the quantum multiplication by $\sigma$ is given by the following formulas:
$$\tau'_2\sigma=2\sigma_4+q, \quad \sigma_2\sigma=\sigma_4, \quad 
\tau'_3\sigma=\sigma_5+qh, \quad \sigma_3\sigma=2\sigma_5,$$
$$\tau'_4\sigma=\sigma_6+q\tau'_2, \quad \sigma_4\sigma=2\sigma_6+q\tau'_2, \quad 
\tau'_5\sigma=q\tau'_3, \quad \sigma_5\sigma=\sigma_7+2q\tau'_3,$$
$$\sigma_6\sigma=q\tau'_4+q^2, \qquad 
\sigma_7\sigma=q\tau'_5+q^2h.$$
\smallskip

We could also check the generic semi-simplicity of $QA^*(X)$, which in 
\cite{GPPS} was directly deduced from the quantum Chevalley formula. 
\medskip

\section{Cohomology of the $Spin_7$-variety}

Using the fact that the $Spin_7$-variety $X$ is a generic hyperplane 
section of the spinor variety $\Sbb$ for $Spin_{10}$, its cohomology 
is easily described. First observe that the two closed orbits 
in $X$ are quadrics of dimensions $5$ and $6$, so that the 
topological Euler number is $6+8=14$. If we use the Schubert 
basis $(\sigma'_i,\tau_j)$, where the classes $\tau_j$ are induced
from the closed orbit $\Qbb^6$, we get the following
Hasse diagram:

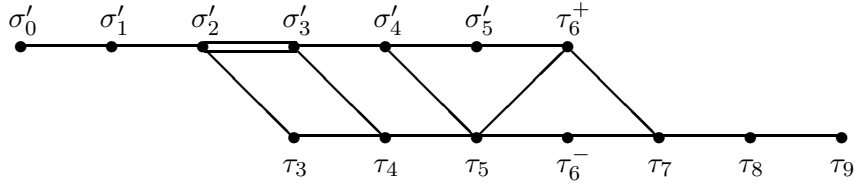
\setlength{\unitlength}{6mm}
\thicklines
\begin{picture}(20,5.3)(-2.5,-1.5)
\multiput(0,2)(2,0){7}{$\bullet$}
\multiput(6,0)(2,0){7}{$\bullet$}

\multiput(0.2,2.2)(2,0){2}{\line(1,0){2}}
\multiput(6.2,2.2)(2,0){3}{\line(1,0){2}}
\put(4.2,2.06){\line(1,0){2}}\put(4.2,2.26){\line(1,0){2}}
\multiput(6.2,.2)(2,0){6}{\line(1,0){2}}
\multiput(4.25,2.05)(2,0){3}{\line(1,-1){1.9}}
\put(12.25,2.05){\line(1,-1){1.9}}
\put(10.25,.25){\line(1,1){1.9}}

\put(-.1,2.6){$\sigma'_0$}\put(1.9,2.6){$\sigma'_1$}
\put(3.9,2.6){$\sigma'_2$}\put(5.9,2.6){$\sigma'_3$}
\put(7.9,2.6){$\sigma'_4$}\put(9.9,2.6){$\sigma'_5$}
\put(11.9,2.6){$\tau^+_6$}

\put(5.9,-.6){$\tau_3$}
\put(7.9,-.6){$\tau_4$}\put(9.9,-.6){$\tau_5$}
\put(11.9,-.6){$\tau_6^-$}
\put(13.9,-.6){$\tau_7$}\put(15.9,-.6){$\tau_8$}
\put(17.9,-.6){$\tau_9$}
\end{picture}

We deduce that the Chow ring is generated by the hyperplane
class $h=\sigma'_1$ and the degree three class $\tau=\tau_3$. 
Moreover, from the Chevalley formula we get
$$\sigma'_2=h^2, \quad \sigma'_3=\frac{h^3-\tau}{2}, \quad 
\tau_4=h\tau, \quad \sigma'_4=\frac{h^4-3h\tau}{2}, \quad 
\tau_5=h^2\tau, \quad \sigma'_5=\frac{h^5-5h^2\tau}{2},$$
$$ \tau_6^+=\frac{h^6-5h^3\tau}{2}, \qquad \tau_6^-=\frac{7h^3\tau-h^6}{2}, \qquad \tau_7=
\frac{h^7}{12}, \quad \tau_8=\frac{h^8}{12}, \quad \tau_9=\frac{h^9}{12}.$$

\medskip
There must be two relations between the generators, in degrees
six and seven. For the latter we can choose $R_7=6h^4\tau-h^7$.
In order to find the former we use the fact that the restriction 
map from the Chow ring of $\Sbb$ to the Chow ring of $X$ is surjective. The Schubert classes in $A^*(\Sbb)$ will be denoted 
$\gamma_{\lambda}$, for $\lambda$ a strict partition with parts smaller than five. In degree three the Schubert classes $\gamma_3$ and $\gamma_{21}$ have degree seven and five, respectively. We deduce that $\tau$ is just the restriction of the difference 
$\gamma_3-\gamma_{21}$. Since $A^6(\Sbb)$ has rank two there must exist a linear relation between $\gamma_3^2$, $\gamma_{21}\gamma_3$ and $\gamma_{21}^2$. By applying the Pieri formulas for the spinor variety we get $\gamma_3^2=2\gamma_{21}^2$. By restricting to $X$ we deduce the 
relation we are looking for, namely 
$$R_6=\tau^2-6h^3\tau+h^6.$$
(Note that it follows that $h\tau^2=0$.) 

This provides enough information to write down the multiplication table. The multiplication by $\tau$ is given by 
$$\sigma'_1\tau=\tau_4, \quad \sigma'_2\tau=\tau_5, \quad \sigma'_3\tau =\tau_6^+, \quad 
\tau_3\tau=\tau_6^--\tau_6^+, \quad \tau_4\tau=0, $$ $$\sigma'_4\tau=\tau_7,\quad 
\tau_5\tau=0, \quad \sigma'_5\tau=\tau_8,\quad \tau_6^+\tau
=\tau_9, \quad \tau_6^-\tau=-\tau_9.$$
\smallskip

\noindent 
And the missing products are the following:
$$(\sigma'_3)^2=\tau_6^++\tau_6^-, \quad \sigma'_3\sigma'_4=\tau_7, \quad  \sigma'_3\tau_4=\tau_7, \quad \sigma'_3\sigma'_5=0,
\quad  \sigma'_3\tau_5=\tau_8, $$
$$\sigma'_3\tau_6^+=0, \quad \sigma'_3\tau_6^-=\tau_9, \quad 
(\sigma'_4)^2=0, \quad \sigma'_4\tau_4=\tau_8, \quad (\tau_4)^2=\tau_8, $$
$$\sigma'_4\sigma'_5=-\tau_9,
\quad \sigma'_4\tau_5=\tau_9, \quad \tau_4\sigma'_5=\tau_9, \quad 
\tau_4\tau_5=0.$$
\smallskip

In terms of the Poincar\'e dual basis we get the reversed Hasse diagram

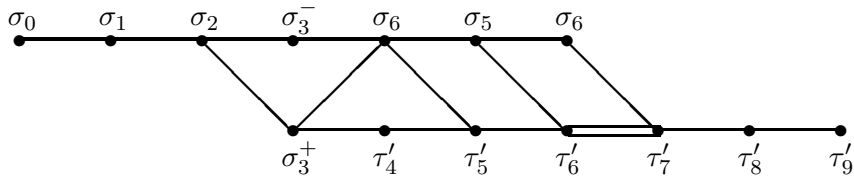
\setlength{\unitlength}{6mm}
\thicklines
\begin{picture}(20,5.3)(-2.5,-1.5)
\multiput(0,2)(2,0){7}{$\bullet$}
\multiput(6,0)(2,0){7}{$\bullet$}

\multiput(0.2,2.2)(2,0){6}{\line(1,0){2}}
\put(12.2,.06){\line(1,0){2}}\put(12.2,.26){\line(1,0){2}}
\multiput(6.2,.2)(2,0){3}{\line(1,0){2}}
\multiput(14.2,.2)(2,0){2}{\line(1,0){2}}
\multiput(8.25,2.05)(2,0){3}{\line(1,-1){1.9}}
\put(4.25,2.05){\line(1,-1){1.9}}
\put(6.25,.25){\line(1,1){1.9}}

\put(-.1,2.5){$\sigma_0$}\put(1.9,2.5){$\sigma_1$}
\put(3.9,2.5){$\sigma_2$}\put(5.9,2.5){$\sigma_3^-$}
\put(7.9,2.5){$\sigma_6$}\put(9.9,2.5){$\sigma_5$}
\put(11.9,2.5){$\sigma_6$}

\put(5.9,-.6){$\sigma^+_3$}
\put(7.9,-.6){$\tau'_4$}\put(9.9,-.6){$\tau'_5$}
\put(11.9,-.6){$\tau'_6$}
\put(13.9,-.6){$\tau'_7$}\put(15.9,-.6){$\tau'_8$}
\put(17.9,-.6){$\tau'_9$}
\end{picture}

Since the index of $X$ is seven, the quantum cohomology is very easy to deduce from the quantum Chevalley formula, computed in 
\cite{GPPS}, Proposition 4.4: quantum corrections do appear only 
for $h\tau_6^-=\tau_7+q$, $h\tau_7=\tau_8+qh$, $h\tau_8=\tau_9+q\sigma'_2$ and $h\tau_9=q\sigma'_3$. 
In particular we get $h(\tau_6^--\tau_6^+)=q$, and since the 
the Giambelli type formulas above are valid in quantum cohomology up to degree six, we deduce that 
$$R_6(q)=R_6, \qquad R_7(q)=6h^4\tau-h^7-q.$$
Moreover the Giambelli type formulas in degree bigger than six must be corrected as
$$\tau_7=
\frac{h^7}{12}-\frac{5}{12}q, \quad \tau_8=\frac{h^8}{12}-\frac{17}{12}qh, \quad \tau_9=\frac{h^9}{12}-\frac{29}{12}qh^2.$$

Finally, the quantum multiplication by the generator $\tau$ is given, in degree bigger than six, by the following formulas:
$$\tau_4\tau=q, \quad \sigma'_4\tau=\tau_7-q,\quad 
\tau_5\tau=qh, \quad \sigma'_5\tau=\tau_8-qh,\quad \tau_6^+\tau
=\tau_9,$$
$$\tau_6^-\tau=-\tau_9+qh^2, \quad \tau_7\tau=q\sigma'_3, \quad 
\tau_8\tau=q\sigma'_4, \quad \tau_9\tau=q\sigma'_5.$$

\bigskip

\bibliographystyle{amsalpha}
\bibliography{two_orbits_sections}

\medskip 

\end{document}